\documentclass[12pt,a4paper]{article}
\usepackage{etex}
\usepackage{amsmath,amssymb,latexsym,amsthm}
\usepackage{authblk}
\usepackage{color}

\theoremstyle{plain}
\newtheorem{theorem}{Theorem}[section]

\newtheorem{lemma}[theorem]{Lemma}
\newtheorem{corollary}[theorem]{Corollary}

\theoremstyle{definition}

\theoremstyle{remark}

\title{\bf Characterizing $\boldsymbol{2}$-Distance Graphs \\ and Solving the Equations \\ $\boldsymbol{T_2(X)=kP_2}$ or $\boldsymbol{K_m \cup K_n}$}
\author{{\bf Ramuel P. Ching}\thanks{The author is supported by the Accelerated Science and Technology Human Resource Development Program of the Department of Science and Technology - Science Education Institute (Philippines).} \\ {\bf I.J.L. Garces}\thanks{Corresponding author: ijlgarces@ateneo.edu}}
\affil{Department of Mathematics \\ Ateneo de Manila University \\ Loyola Heights, Quezon City \\ The Philippines}
\date{}

\begin{document}

\maketitle

\begin{abstract}
Let $X$ be a finite, simple graph with vertex set $V(X)$. The \mbox{$2$-distance} graph $T_2(X)$ of $X$ is the graph with the same vertex set as $X$ and two vertices are adjacent if and only if their distance in $X$ is exactly $2$. A graph $G$ is a $2$-distance graph if there exists a graph $X$ such that $T_2(X)=G$. In this paper, we give three characterizations of $2$-distance graphs, and find all graphs $X$ such that $T_2(X)=kP_2$ or $K_m \cup K_n$, where $k \ge 2$ is an integer, $P_2$ is the path of order $2$, and $K_m$ is the complete graph of order $m \ge 1$.

\medskip\noindent
2010 Mathematics Subject Classification: Primary 05C12

\noindent
Keywords: Distance in graphs, $2$-distance graph, complement of a graph, cocktail party graph, bipartite graph, complete graph
\end{abstract}


\section{Introduction}
By a graph, we mean a finite, simple, but not necessarily connected graph. Let $X$ be a graph with vertex set $V(X)$ and edge set $E(X)$. For each $u \in V(X)$, let $N_X(u)$ be the neighborhood of $u$ in $X$ (that is, the set of all vertices of $X$ that are adjacent to $u$), and the cardinality of $N_X(u)$ is the \textit{degree} of $u$ in $X$, which is denoted by $\text{deg}_X(u)$.

The \textit{distance} between two vertices $u$ and $v$ in $X$, denoted by $\text{dist}_X(u,v)$, is the length of a shortest $u$-$v$ path in $X$. For convenience, we set $\text{dist}_X(u,v)=0$ if and only if $u=v$, and  $\text{dist}_X(u,v)=\infty$ if and only if $u$ and $v$ are in different components of $X$. The number $\text{diam}(X)=\max\{\text{dist}_X(u,v)\,:\,u,v\in V(X)\}$ is the \textit{diameter} of $X$.

The \textit{$2$-distance graph} of $X$, denoted by $T_2(X)$, is the graph with $V(T_2(X))=V(X)$ in which two vertices $u$ and $v$ are adjacent in $T_2(X)$ if and only if $\text{dist}_X(u,v)=2$. A graph $G$ is a \textit{$2$-distance graph} if there exists a graph $X$ such that $T_2(X)=G$, where equality refers to graph isomorphism. Furthermore, we say that a graph $X_0$ is a \textit{solution} to the equation $T_2(X)=G$ if $T_2(X_0)=G$.

The idea of $2$-distance graph is a particular case of a general notion of $k$-distance graph $T_k(X)$, which was first studied by Harary, Hoede, and Kedlacek \cite{harary}. They investigated the connectedness of $2$-distance graph. In the book by Prisner \cite[pp157-159]{prisner}, the dynamics of $k$-distance operator was explored. Furthermore, Boland, Haynes, and Lawson \cite{boland} extended the $k$-distance operator to a graph invariant they call distance-$n$ domination number. Recently, Azimi and Farrokhi \cite{azimi} studied all graphs whose $2$-distance graphs have maximum degree $2$. They also solved the problem of finding all graphs whose $2$-distance graphs are paths or cycles.

In this paper, we give three characterizations of $2$-distance graphs, and find all graphs $X$ such that $T_2(X)=kP_2$ or $K_m \cup K_n$, where $k \ge 2$ is an integer, $P_2$ is the path of order $2$, and $K_m$ is the complete graph of order $m \ge 1$. We end the paper with some open problems.

Given a graph $X$, we denote its complement and its maximum degree by $X^c$ and $\Delta(G)$, respectively. Other graph-theoretic terms and notations that are not explicitly defined here can be found in \cite{Bondy&Murty}.

\section{Characterizations of $\boldsymbol{2}$-distance graphs}
We now present the first of the three characterizations of $2$-distance graphs.

\begin{theorem}\label{theorem:1stcharacterization}
Let $G$ be a graph. The following properties are equivalent:
\begin{enumerate}
\item[{\rm (i)}] $G$ is a $2$-distance graph;
\item[{\rm (ii)}] for every $v_1v_2\in E(G)$, there exists $v_3\in V(G)$ that is not adjacent to both $v_1$ and $v_2$ in $G$; and
\item[{\rm (iii)}] ${\rm diam}(G^c)\leq 2$.
\end{enumerate}
\end{theorem}

\begin{proof}
(i) $\Rightarrow$ (ii): Suppose $G=T_2(X)$ for some graph $X$. Let $v_1$ and $v_2$ be adjacent vertices in $G$. Then $\text{dist}_X(v_1,v_2)=2$. Thus, there exists a vertex $v_3$ adjacent to both $v_1$ and $v_2$ in $X$. This implies that $v_3$ is not adjacent to both $v_1$ and $v_2$ in $G$.

(ii) $\Rightarrow$ (iii): We consider two cases.

{\sc Case 1.} Suppose $G^c$ is a complete graph. Then $\text{diam}(G^c)=0$ or $1$.

{\sc Case 2.} Suppose there exist two vertices in $G^c$ that are not adjacent, say $v_1$ and $v_2$. It implies that $v_1$ and $v_2$ are adjacent in $G$, and, by assumption, there exists a vertex $v_3$ not adjacent to both of them in $G$. It follows that $v_1$ and $v_2$ are both adjacent to $v_3$ in $G^c$, and so $\text{dist}_{G^c}(v_1,v_2)=2$. Therefore, we have $\text{diam}(G^c) \le 2$.

(iii) $\Rightarrow$ (i): If $\text{diam}(G^c)=0$ or $1$, then $G^c$ is a complete graph, and so $G$ is the empty graph. This implies that $T_2(G^c)=G$. On the other hand, if $\text{diam}(G^c)=2$, then the vertices that are adjacent in $T_2(G^c)$ are exactly those vertices that are not adjacent in $G^c$. Thus, we get $T_2(G^c)=(G^c)^c=G$.
\end{proof}

The following corollaries are quick consequences of the preceding theorem.

\begin{corollary}
Every disconnected graph is a $2$-distance graph.
\end{corollary}

\begin{corollary}\label{corollary:complement}
Let $G$ be a $2$-distance graph. Then $G^c$ is a solution to the equation $T_2(X)=G$. Moreover, if $X_0$ is a solution to $T_2(X)=G$, then $E(X_0) \subseteq E(G^c)$.
\end{corollary}

Our second characterization of $2$-distance graphs utilizes a result by Bloom, Kennedy, and Quintas \cite{bloom} that characterizes graphs with diameter $2$.

Let $u$ be a vertex of a graph $X$. A \textit{star centered at $u$} is a subgraph of $X$ consisting of edges that have $u$ as a common vertex. Let $uv$ be an edge of $X$. A \textit{double-star on $uv$}, or simply a \textit{double-star}, is a maximal tree in $X$ that is the union of stars centered at $u$ or $v$ such that both stars contain $uv$.

Let $Y$ be a subgraph of a graph $X$. We say that $Y$ \textit{spans} $X$ if $V(Y)=V(X)$.

\begin{lemma} {\rm \cite[Theorem 1]{bloom}}\label{theorem:bloom}
A graph $G$ has diameter two if and only if $G^c$ is not empty and $G^c$ is not spanned by a double-star.
\end{lemma}

\begin{theorem}
A graph $G$ is a $2$-distance graph if and only if $G$ is not spanned by a double-star.
\end{theorem}

\begin{proof}
Assume that $G$ is a $2$-distance graph.  By Theorem \ref{theorem:1stcharacterization}, it follows that $\text{diam}(G^c)\leq 2$. If diam$(G^c)=0$ or $1$, then $G$ is empty, and so $G$ is not spanned by a double-star. If diam$(G^c)=2$, then, by Lemma \ref{theorem:bloom}, $G$ is not spanned by a double-star.

Assume that $G$ is not spanned by a double-star. If $G$ is not empty, then, by Lemma \ref{theorem:bloom}, we have diam$(G^c)=2$, and so $G$ is a $2$-distance graph by Theorem \ref{theorem:1stcharacterization}. If $G$ is empty, then $G^c$ is a complete graph, and so $\text{diam}(G^c)=0$ or $1$, and, by Theorem \ref{theorem:1stcharacterization} again, $G$ is a $2$-distance graph.
\end{proof}

Finally, we give the third characterization of $2$-distance graphs, which emphasizes the degrees of the vertices.

We first mention the following lemma, whose proof is immediate.

\begin{lemma}\label{lemma:degree}
Let $G$ be a graph with at least one edge, and let $ab \in E(G)$. If $\text{deg}_G(a)+\text{deg}_G(b) < |V(G)|$, then $V(G)\setminus(N_G(a)\cup N_G(b)) \ne \emptyset$.
\end{lemma}

\begin{theorem}\label{theorem:3rdcharacterization}
A graph $G$ is a $2$-distance graph if and only if, for each $v\in V(G)$ with $\text{deg}_G(v)\geq \frac{1}{2}|V(G)|$ and for each $a\in N_G(v)$, there exists $b\in V(G)$ such that $b\notin N_G(v)\cup N_G(a)$.
\end{theorem}

\begin{proof}
Without loss of generality, we assume that $G$ has at least one edge.

Suppose $G$ is a $2$-distance graph. Then the conclusion follows at once from Theorem \ref{theorem:1stcharacterization}.

Conversely, suppose that, for each $v\in V(G)$ with $\text{deg}_G(v)\geq \frac{1}{2}|V(G)|$ and for each $a\in N_G(v)$, there exists a vertex that is not adjacent to both $v$ and $a$. Let $x,y \in V(G)$. If $xy \not\in E(G)$, then $\text{dist}_{G^c}(x,y)=1$. Now, suppose $xy \in E(G)$. We consider two cases.

{\sc Case 1.} If $\text{deg}_G(x) \geq \frac{1}{2}|V(G)|$ or $\text{deg}_G(y) \geq \frac{1}{2}|V(G)|$, then $\text{dist}_{G^c}(x,y)=2$ by assumption.

{\sc Case 2.} If $\text{deg}_G(x) < \frac{1}{2}|V(G)|$ and $\text{deg}_G(y) < \frac{1}{2}|V(G)|$, then $\text{deg}_G(x)+\text{deg}_G(y) < |V(G)|$. By Lemma \ref{lemma:degree}, we have $\text{dist}_{G^c}(x,y)=2$.

Thus, we have $\text{diam}(G^c) = 2$, and, by Theorem \ref{theorem:1stcharacterization}, $G$ is a $2$-distance graph.
\end{proof}

Drawing ideas from the proof of the preceding theorem, the following corollaries follow immediately.

\begin{corollary}
Let $G$ be a graph with $\Delta(G) < \frac{1}{2}|V(G)|$. Then $G$ is a $2$-distance graph.
\end{corollary}

\begin{corollary}
Let $G$ be a graph with at least one edge. Then $G$ is a $2$-distance graph if and only if ${\rm diam}(G^c)=2$.
\end{corollary}

\section{Solutions to the Equation $\boldsymbol{T_2(X)=kP_2,\ k \ge 2}$}

A quick application of Theorem \ref{theorem:1stcharacterization} shows that the path $P_2$ is not a $2$-distance graph. However, the union of at least two paths $P_2$ is a $2$-distance graph. In this section, we find all graphs that satisfy the equation $T_2(X)=kP_2$ for $k \ge 2$.

Let $K_{2n}$ be the complete graph of order $2n$, $n \ge 2$, with $$V(K_{2n})=\{u_1,u_2,\ldots,u_n,v_1,v_2,\ldots,v_n\}.$$
A \textit{cocktail party graph} of order $2n$, denoted by $CP_{2n}$, is the graph with $$V(CP_{2n})=V(K_{2n}) \quad\mbox{and}\quad E(CP_{2n})=E(K_{2n}) \setminus \{u_iv_i:i=1,2,\ldots,n\}.$$

It is not difficult to check that $T_2(P_4)=2P_2$, where $P_4$ is the path of order $4$, and $T_2(CP_{2n})=nP_2$. In general, if $X=pP_4 \cup qCP_{2n}$, where $p \ge 0$, $q \ge 0$, and $n \ge 2$ are integers, then $T_2(X)=(2p+qn)P_2$.

\begin{theorem}
Let $X$ be a connected graph such that $T_2(X)=kP_2$ for some integer $k \ge 2$. Then
$$X=\left\{
\begin{array}{ll}
P_4 \mbox{ or } CP_4 & \mbox{if $k=2$} \\
CP_{2k} & \mbox{if $k \ge 3$.}
\end{array} \right.$$
\end{theorem}

\begin{proof}
As observed in the previous paragraph, we know that $T_2(P_4)=2P_2$ and $T_2(CP_{2k})=kP_2$ for any $k \ge 2$.

For $k=2$, it is not difficult to determine that the only possible solutions are $P_4$ and $CP_4$.

Suppose $k \ge 3$, and let $V(kP_2)=\{u_1,u_2,\ldots,u_k,v_1,v_2,\ldots,v_k\}$ and $E(kP_2)=\{u_iv_i:i=1,2,\ldots,k\}$. We prove that $X=CP_{2k}$ is the only connected graph that satisfies the property.

For convenience, we only show here that $u_1$ must be adjacent to $u_2$ in $X$, and argue similarly that $u_1$ must also be adjacent to all other vertices (except $v_1$).

Suppose, on the contrary, that $u_1u_2 \not\in E(X)$. Then $\text{dist}_X(u_1,u_2) \ge 3$. Because $u_1$ is the only vertex in $X$ that is of distance $2$ from $v_1$, we have $\text{dist}_X(u_1,u_2)=3$. That is, for some vertex $v$, there is an induced path of the form $u_1vv_1u_2$. Since $v_2$ is the only vertex that is of distance $2$ from $u_2$, we must have $v=v_2$. Because $X$ is connected and $|V(X)| \ge 6$, we can find another vertex $w$ that is adjacent to one but not all of the vertices $u_1$, $v_2$, $v_1$, and $u_2$. However, in such scenario, one of $u_1$, $v_2$, $v_1$, and $u_2$ will be of distance $2$ from $w$, which is a contradiction. Thus, $u_1$ is adjacent to $u_2$ in $X$.

In general, it can be shown in a similar fashion that every vertex $u_i$ (respectively, $v_i$) is adjacent to all other vertices in $X$ except $v_i$ (respectively, $u_i$), which implies that $X=CP_{2k}$.
\end{proof}

The following corollary, whose proof follows from the preceding theorem by separately considering each component of a disconnected graph as a connected graph, enumerates all solutions to the equation $T_2(X)=kP_2$ for $k \ge 2$.

\begin{corollary}
Let $k$ be an integer greater than or equal to $2$. Then the solutions to the equation $T_2(X)=kP_2$ are $X=pP_4 \cup qCP_{2n}$ for all integers $p \ge 0$, $q \ge 0$, and $n \ge 2$ that satisfy the Diophantine equation $2p+qn=k$.
\end{corollary}

\section{Solutions to the Equation $\boldsymbol{T_2(G)=K_m\cup K_n}$}
We know from Theorem \ref{theorem:1stcharacterization} that $K_m \cup K_n$ is a $2$-distance graph for any positive integers $m$ and $n$. In this section, we solve the equation $T_2(X)=K_m \cup K_n$.

A graph is said to be \textit{bipartite} if its vertex set can be partitioned into two subsets $A$ and $B$ such that each edge has one endvertex in $A$ and the other in $B$. If every vertex in $A$ is adjacent to every vertex in $B$, then the graph is a \textit{complete bipartite graph}, which is denoted by $K_{m,n}$, where $|A|=m$ and $|B|=n$.

We observe that being bipartite is a hereditary property; that is, if $G$ is bipartite, then all spanning subgraphs of $G$ are also bipartite, and the sizes of the partitions are also inherited.

The following lemma can be shown easily.

\begin{lemma}\label{lemma:complete1}
If $X=K_{1,1}$ or $X=K_{1,1}^c$, then $T_2(X)=K_1 \cup K_1$.
\end{lemma}

\begin{lemma}\label{lemma:complete2}
Let $X$ be a bipartite graph with partitions of sizes $m$ and $n$. If $2 \le \text{\rm diam}(X) \le 3$, then $T_2(X)=K_m \cup K_n$.
\end{lemma}

\begin{proof}
When $\text{diam}(X)=2$, it is not difficult to see that $X=K_{m,n}$, which immediately implies that $T_2(X)=K_m \cup K_n$.

Suppose that $\text{diam}(X)=3$, and let the partitions of $V(X)$ be $A$ and $B$ with $|A|=m>1$ and $|B|=n>1$. Let $x$ and $y$ be distinct vertices of $X$, and we consider two cases.

{\sc Case 1.} If $x,y \in A$ or $x,y \in B$, then $\text{dist}_X(x,y)$ must be of even parity. Since $\text{diam}(X)=3$, we must have $\text{dist}_X(x,y)=2$, and so $xy \in T_2(X)$.

{\sc Case 2.} If $x \in A$ and $y \in B$ (or $y \in A$ and $x \in B$), then $\text{dist}_X(x,y)$ must be of odd parity, and so $xy \not\in T_2(X)$.

Therefore, it follows that if $\text{diam}(X)=3$, then $T_2(X)=K_m \cup K_n$.
\end{proof}

\begin{lemma}\label{lemma:complete3}
Let $m$ and $n$ be positive integers with at least one of them greater than $1$. If $T_2(X)=K_m \cup K_n$, then $X$ is a nonempty spanning subgraph of $K_{m,n}$ (that is, $X$ is a nonempty bipartite graph with partitions of sizes $m$ and $n$).
\end{lemma}

\begin{proof}
By definition, the vertices of $K_m$ are not adjacent to each other in $X$. The same is true for the vertices of $K_n$. The vertices of $K_m$ and $K_n$ give rise to a partitioning of $V(X)$ into two sets of sizes $m$ and $n$, and any edge joins one vertex in one set and another vertex in the other set. Since one of $m$ and $n$ is greater than $1$, $K_m \cup K_n$ is nonempty, and so $X$ is also nonempty. Thus, $X$ is a nonempty bipartite graph with partitions of sizes $m$ and $n$.
\end{proof}

\begin{theorem}
Let $m$ and $n$ be positive integers. Then
\begin{enumerate}
\item[{\rm (i)}] $T_2(X)=K_1 \cup K_1$ if and only if $X=K_{1,1}$ or $X=K_{1,1}^c$; and
\item[{\rm (ii)}] $T_2(X)=K_m \cup K_n$, where $m>1$ or $n>1$, if and only if $X$ is a spanning subgraph of $K_{m,n}$ such that $2 \le \text{\rm diam}(X) \le 3$.
\end{enumerate}
\end{theorem}

\begin{proof} The proof of (i) is straightforward with Lemma \ref{lemma:complete1}. We now prove (ii).

If $X$ is a spanning subgraph of $K_{m,n}$ such that $2 \le \text{diam}(X) \le 3$, then, by Lemma \ref{lemma:complete2}, $T_2(X)=K_m \cup K_n$.

Suppose that $T_2(X)=K_m \cup K_n$, where $m>1$ or $n>1$. By Corollary \ref{corollary:complement} and Lemma \ref{lemma:complete3}, we know that $(K_m \cup K_n)^c=K_{m,n}$ is a solution to the given equation, and $X$ is a spanning (bipartite) subgraph of $K_{m,n}$ with at least one edge. Since $\text{diam}(K_{m,n})=2$, it follows that $\text{diam}(X) \ge 2$.

Suppose $\text{diam}(X) \ge 4$. Then there exist two vertices $x$ and $y$ such that $\text{dist}_X(x,y)=4$. This implies that $x$ and $y$ belong to one partition of $V(X)$, but they are not of distance $2$ in $X$, which contradicts the fact that they are adjacent in $T_2(X)$ since $T_2(X)=K_m \cup K_n$.

Therefore, we have $\text{diam}(X) \le 3$, which completes the proof of (ii).
\end{proof}

The following corollaries are quick consequences of the preceding theorem.

\begin{corollary}
Let $X_0$ be a connected bipartite graph with partitions $A$ and $B$ of sizes $m$ and $n$, respectively. If there exist a vertex in $A$ of degree $n$ and a vertex in $B$ of degree $m$, then $T_2(X_0)=K_m \cup K_n$. Moreover, any solution to the equation $T_2(X)=K_m\cup K_n$ is a subgraph of a graph having the same property as that of $X_0$.
\end{corollary}

\begin{corollary}\label{corollary:star}
Let $n \ge 2$ be a positive integer. Then $T_2(X)=K_1 \cup K_n$ if and only if $X=(K_1 \cup K_n)^c$.
\end{corollary}

\end{document}